\documentclass[12pt,a4paper]{amsart}
\usepackage{amsfonts}
\usepackage{amssymb}
\usepackage{amsmath}
\usepackage{amsthm} 
\usepackage{cite}
\usepackage{enumitem}
\usepackage{tikz}
\usepackage{subfigure}
\usepackage{latexsym}
\usepackage{dsfont}

\theoremstyle{plain}
\newtheorem{thm}{Theorem}

\newtheorem{lem}{Lemma}
\newtheorem{prop}{Proposition}
\newtheorem{cor}{Corollary}
\newtheorem{rem}{Remark}

\newcommand{\vecp}[1]{{\vec #1}\,'}
\providecommand{\ind}{\mathds{1}} 
\providecommand{\les}{\lesssim}

\providecommand{\N}{\mathbb{N}}
\providecommand{\R}{\mathbb{R}}
\providecommand{\Z}{\mathbb{Z}}

\providecommand{\ov}{\overline}

\DeclareMathOperator{\supp}{supp}

\setitemize{itemsep=+2pt} 
\setenumerate{itemsep=+2pt}

\begin{document}


\allowdisplaybreaks

\title{Real Interpolation for mixed Lorentz spaces and Minkowski's inequality}

\author{Rainer Mandel\textsuperscript{1}}
\address{\textsuperscript{1}Karlsruhe
Institute of Technology, Institute for Analysis, Englerstra{\ss}e 2, 76131 Karlsruhe, Germany}
\email{Rainer.Mandel@kit.edu}


\begin{abstract}
  We prove embeddings and identities for real interpolation spaces between mixed Lorentz spaces.
  This partly relies on Minkowski's (reverse) integral inequality in Lorentz spaces
  $L^{p,r}(X)$  under optimal assumptions on the exponents $(p,r)\in
  (0,\infty)\times (0,\infty]$.
\end{abstract}

\maketitle
\allowdisplaybreaks

\section{Introduction}

  
  The characterization of real interpolation spaces between mixed Lorentz spaces is an unsolved problem.
  As we show in a forthcoming paper~\cite{Mandel2023}, such spaces show up in the analysis of the
  (Fourier) Restriction-Extension operator for the sphere $f\mapsto \mathcal F^{-1}(\hat f\,d\sigma)$ when the
  latter is considered as a linear map acting on $O(d-k)\times O(k)$-symmetric functions belonging to
  $L^p(\R^d)$ where $k\in\{2,\ldots,d-2\},d\geq 4$. Another application in the context of generalized
  Hardy-Littlewood inequalities was recently given by Chen and Sun~\cite{ChenSun2020}. In order to determine the optimal mapping properties of
  linear operators, it is often reasonable to make use of interpolation theory.
  Our aim is to provide new embeddings and identities for the real interpolation spaces between mixed Lorentz
  spaces that extend the results from~\cite{ChenSun2020}. For other interpolation methods see
  \cite{Milman1981} and, in the more general context of mixed norm spaces, Section~3.2 in
  \cite{WuYangYuan2022}. Our first main result shows that the characterization of these spaces
  is equivalent to the computation of real interpolation spaces for mixed Lebesgue spaces, which is a priori
  easier. In relevant special cases this even permits to characterize these interpolation spaces explicitly,
  see Corollary~\ref{cor:MixedInterpolation} below.
  For instance we shall obtain the identity
  \begin{align} \label{eq:example}
    \begin{aligned}
    (L^{p_0,r}(\R^{d-k};L^{p_0,r}(\R^k)), L^{p_1}(\R^d))_{\theta,q}
     = L^{p_\theta,q}(\R^d), \qquad\quad \frac{1}{p_\theta}= \frac{1-\theta}{p_0}+\frac{\theta}{p_1}
     \end{aligned} 
  \end{align}
  provided that $1<p_0\neq p_1<\infty$, $0<\theta<1$, $r,q\in [1,\infty]$. This is a nontrivial statement given that  
  the mixed Lorentz space $L^{p_0,r}(\R^{d-k};L^{p_0,r}(\R^k))$ does not coincide with 
  $L^{p_0,r}(\R^d)$ unless $r=p_0$, see~\cite[p.287]{Cwikel1974}. In what follows let $(X,\mu_X)$ be a 
  $\sigma-$finite measure space equipped with a suitable $\sigma$-algebra. In the context of our first main result we
  additionally assume this measure space to be  non-atomic or discrete, i.e., purely atomic with at
  most countably many atoms of equal measure. This will allow to use some results of
  Blozinski~\cite{Blozinski1981} later on. The standard examples are (weighted) Euclidean spaces or $\N$ or
  $\Z$ equipped with the counting measure.
  We write $g\in L^{p,r}(X)$ provided
  that $g:X\to \ov\R$ is measurable and the following quasi-norm is finite $$ \|g\|_{p,r}
    :=  \|g\|_{L^{p,r}(X)}
   := \| \rho\, \mu_X(\{|g|\geq \rho\})^{\frac{1}{p}}\|_{L^r(\R_+,\rho^{-1}\,d\rho)}.
  $$
  We recall that $\|\cdot\|_{p,r}$ is equivalent to a norm on
  $L^{p,r}(X)$  provided that
  \begin{equation} \label{eq:NormabilityCondition}
    1<p<\infty,1\leq r\leq \infty
    \quad\text{or}\quad p=r=1  \quad\text{or}\quad p=r=\infty.
  \end{equation} 
  In this case $L^{p,r}(X)$ is said to be normable. Let $(Y,\mu_Y)$ denote a second  measure space with
  analogous properties and define the product measure $\mu:=\mu_X\otimes \mu_Y$ on the corresponding product
  measurable space. Note that the $\sigma$-finiteness of $(X,\mu_X),(Y,\mu_Y)$  implies that $\mu$ 
  is well-defined. For $\vec p:=(p^1,p^2),\vec r:=(r^1,r^2)$ the mixed Lorentz space $L^{\vec p,\vec r} :=
  L^{p^1,r^1}(X;L^{p^2,r^2}(Y))$ consists of all
  functions $f:X\times Y\to\ov\R$ that are measurable and satisfy 
  $$
    \|f\|_{\vec p,\vec r} := \| F\|_{L^{p^1,r^1}(X)} <\infty
    \quad\text{where }F(x):= \|f(x,\cdot)\|_{L^{p^2,r^2}(Y)} \text{ for }x\in X.
  $$   
  Our first main result reads as follows.
   
\begin{thm}\label{thm:MixedInterpolationI}
  Let the measure spaces $(X,\mu_X),(Y,\mu_Y)$ be $\sigma$-finite and non-atomic or discrete. Moreover 
  assume that each of the tuples $(p_0^i,r_0^i),(p_1^i,r_1^i)$ satisfies $p_0^i\neq p_1^i$ and \eqref{eq:NormabilityCondition} for $i=1,2$.
  Then we have for $0<\theta_0\neq\theta_1<1,0<\vartheta<1$ and all $q\in [1,\infty]$ 
   \begin{align} \label{eq:inclusion}
     (L^{\vec p_0,\vec r_0},L^{\vec p_1,\vec r_1})_{(1-\vartheta)\theta_0+\vartheta\theta_1,q} 
    = (L^{\vec p_{\theta_0}}, L^{\vec p_{\theta_1}})_{\vartheta,q}.
   \end{align}
   where $\frac{1}{p_{\theta_j}^i} = \frac{1-\theta_j}{p_0^i}+\frac{\theta_j}{p_1^i}$ for $i=1,2$ and $j=0,1$. 
\end{thm}

 The relation~\eqref{eq:inclusion} is particularly interesting when $L^{\vec p_{\theta_0}},L^{\vec
 p_{\theta_1}}$ coincide with ordinary Lebesgue spaces over $X\times Y$. By the Tonelli-Fubini
 Theorem this is the case if $p_0^1=p_0^2$ and $p_1^1=p_1^2$. Then the interpolation space on the right is in
 fact an ordinary Lorentz space over $X\times Y$, written $L^{p_\theta,q}$. This observation can be stated as
 follows.

\begin{cor} \label{cor:MixedInterpolation}
   In addition to the hypotheses of Theorem~\ref{thm:MixedInterpolationI} assume $p_0^1=p_0^2,p_1^1=p_1^2$.
   Then we have for all $\theta\in (0,1)$ and $q\in [1,\infty]$ 
   $$
      (L^{\vec p_0,\vec r_0},L^{\vec p_1,\vec r_1})_{\theta,q} = L^{p_\theta,q}   
   $$
   where $\frac{1}{p_\theta} = \frac{1-\theta}{p_0^i}+\frac{\theta}{p_1^i}$ for $i=1,2$. 
   In particular, \eqref{eq:example} holds.
\end{cor}

 The case  $(p_0^i,r_0^i)=(p_1^i,r_1^i)$ for one $i\in\{1,2\}$ is not covered by
 Theorem~\ref{thm:MixedInterpolationI}, so what can be said? For $i=2$, when the ``second Lorentz spaces''
 over $Y$ coincide, there is a complete and simple answer that extends the observation by
 Cwikel~\cite[p.286]{Cwikel1974} about this topic: Corollary~4.5 in \cite{ChenSun2022} shows $(B_0(A),B_1(A))_{\theta,q} =
 (B_0,B_1)_{\theta,q}(A)$ under very mild assumptions on the couple $(B_0,B_1)$ and $A$. In
 particular, this shows that real interpolation pairs of mixed Lorentz spaces can also be mixed Lorentz spaces.
 The case $i=1$ is much more challenging. 
 To treat this case in a more abstract setting we shall be interested in Lorentz-space versions of the
 embeddings  
   \begin{align}
     (\mathcal L^p(X;A_0),\mathcal L^p(X;A_1))_{\theta,q} &\subset \mathcal L^p(X;(A_0,A_1)_{\theta,q})
     &&\text{if }1\leq q\leq p\leq \infty,
      \label{eq:CwikelInclusionI}   \\
     (\mathcal L^p(X;A_0),\mathcal L^p(X;A_1))_{\theta,q} &\supset \mathcal L^p(X;(A_0,A_1)_{\theta,q})
     &&\text{if } 1\leq p\leq q\leq \infty,
     \label{eq:CwikelInclusionII}
   \end{align}
   which are due to Cwikel~\cite[p.288]{Cwikel1974}. Here, $(A_0,A_1)$ denotes a compatible couple
   in the sense of \cite[Section~2.3]{BerghLoefstroem1976} and $\mathcal L^p(X;A_0)$ stands for a Bochner
   space that slightly differs from $L^p(X;A)$ due to some subtle measurability issues, see Section~3 in
   \cite{ChenSun2022}.
  We recall that the real interpolation space $(A_0,A_1)_{\theta,q}$ comes with the quasi-norm
  \begin{equation}\label{eq:DefNormInterpolationspace}
    \|g\|_{(A_0,A_1)_{\theta,q}}
    = \left(\int_0^\infty s^{-\theta q-1} K(s,g,A_0,A_1)^q \,ds \right)^{\frac{1}{q}}
  \end{equation}
  where $0<\theta<1,0<q\leq \infty$ and Peetre's $K$-functional is given by
  \begin{equation*}
    K(s,g,A_0,A_1) = \inf_{g_0+g_1=g} \|g_0\|_{A_0}+s\|g_1\|_{A_1}
  \end{equation*}
  In~\cite{Cwikel1974} Cwikel states that the embeddings
  \eqref{eq:CwikelInclusionI},\eqref{eq:CwikelInclusionII} result from a straightforward application of
  Minkowski's inequality in integral form respectively its reverse version. It is therefore not
  surprising that a generalization of these embeddings to Lorentz spaces requires for a 
  complete understanding of Minkowski's inequality in this setting, which we shall provide in
  Section~\ref{sec:Minkowski} and~\ref{sec:MinkowskiCounter}. In this way we
  obtain the following generalization of \eqref{eq:CwikelInclusionI},\eqref{eq:CwikelInclusionII}:

%
%

\begin{thm} \label{thm:MixedInterpolationII}  
  Let  $(X,\mu_X)$ be a  $\sigma$-finite measure space and $(A_0,A_1)$ a compatible couple of Banach spaces. 
  Then the following holds:
  \begin{itemize}
    \item[(i)] If $0<q<p<\infty, q\leq r\leq \infty$ or $0<p=q=r\leq \infty$
    then
  $$
    (L^{p,r}(X;A_0),L^{p,r}(X;A_1))_{\theta,q}
    \subset L^{p,r}(X;(A_0,A_1)_{\theta,q}).
  $$
    \item[(ii)]  If $0<p< q\leq \infty,0<r\leq q$ or $0<p=q=r\leq \infty$
    then
  $$
    (L^{p,r}(X;A_0),L^{p,r}(X;A_1))_{\theta,q}
    \supset L^{p,r}(X;(A_0,A_1)_{\theta,q}).
  $$
  \end{itemize}
\end{thm}
  
  Given the optimality of the exponent range for Minkowski's (reverse) inequality that we give below and its 
  direct application in the proof of Theorem~\ref{thm:MixedInterpolationII}, we expect that the embeddings
  (i),(ii) are optimal without further assumptions on $X,A_0,A_1$. In the following we write  $a\les
  b$ instead of $a\leq C b$ for some $C>0$ only depending on fixed parameters. Similarly for $\gtrsim$ and
  $a\simeq b$ meaning $a\les b$ and $b\les a$.

\section{Proof of Theorem~\ref{thm:MixedInterpolationI} and Corollary~\ref{cor:MixedInterpolation}}

  
  \begin{prop}\label{prop:dual}
    Let the measure spaces $(X,\mu_X),(Y,\mu_Y)$ be $\sigma$-finite and non-atomic or discrete. Moreover 
    assume $p^i=r^i=1$ or $1<p^i<\infty,1\leq r^i<\infty$ for each $i\in\{1,2\}$.
    Then $(L^{\vec p,\vec r})' = L^{\vecp p,\vecp r}$ 
    where $\vecp p,\vecp r$ denotes the corresponding tuple of H\"older conjugates. 
  \end{prop}
  \begin{proof}
    The conditions ensure that the Lorentz spaces $L^{p^1,r^1}(X)$ and $L^{p^2,r^2}(Y)$ have absolutely
    continuous norm \cite[Theorem~8.5.1]{PickKufner2013} and so we have the characterization of the dual
    spaces given by $(L^{p^1,r^1}(X))' = L^{(p^1)',(r^1)'}(X), (L^{p^2,r^2}(Y))' = L^{(p^2)',(r^2)'}(Y)$, see
    \cite[Corollary~8.5.2]{PickKufner2013}.
    Given that $X$ is non-atomic or discrete, Theorem~3.2.II in \cite{Blozinski1981} implies 
    $$
      (L^{\vec p,\vec r})' 
      = (L^{p^1,r^1}(X))'(L^{p^2,r^2}(Y))'
      =  L^{(p^1)',(r^1)'}(X)(L^{(p^2)',(r^2)'}(Y))
      = L^{\vecp p,\vecp r}.
    $$    
    Note that the second and third expression are written in the notation from \cite{Blozinski1981}
  \end{proof}
  
 With this at hand  we can prove Theorem~\ref{thm:MixedInterpolationI}. We shall occasionally write
 $\vec 1:=(1,1)$ and  $\vec \infty:=(\infty,\infty)$.
 Let the tuples $(p_0^i,r_0^i),(p_1^i,r_1^i)$ for $i=1,2$ be given as in the theorem. Then a special case
 of the interpolation inequality from \cite[Theorem~2.22]{ChenSun2020} gives
  \begin{align} \label{eq:ChenSunInequality}
    \|f\|_{L^{\vec p_\theta}}
    \les \|f\|_{L^{\vec p_0,\vec \infty}}^{1-\theta} \|f\|_{L^{\vec p_1,\vec \infty}}^{\theta}
    \qquad\text{for all }\theta\in (0,1)
  \end{align} 
  where $\frac{1}{p_\theta^i}=\frac{1-\theta}{p_0^i}+\frac{\theta}{p_1^i}$ for $i=1,2$.
  So \cite[Theorem~3.5.2]{BerghLoefstroem1976} implies
  \begin{equation*} 
    L^{\vec p_\theta}  \supset (L^{\vec p_0,\vec \infty},L^{\vec p_1,\vec \infty})_{\theta,1}
    \qquad\text{for all }\theta\in (0,1).
  \end{equation*} 
  We fix $0<\theta_0\neq \theta_1<1$ and interpolate the corresponding embeddings  with each
  other. This gives, for $0<\vartheta<1, 1\leq q\leq \infty$,
  \begin{align*} 
    (L^{\vec p_{\theta_0}}, L^{\vec p_{\theta_1}})_{\vartheta,q}
    \supset  ((L^{\vec p_0,\vec \infty},L^{\vec p_1,\vec \infty})_{\theta_0,1},(L^{\vec p_0,\vec
    \infty},L^{\vec p_1,\vec \infty})_{\theta_1,1} )_{\vartheta,q} 
    =  (L^{\vec p_0,\vec \infty},L^{\vec p_1,\vec \infty})_{(1-\vartheta)\theta_0+\vartheta\theta_1,q}
  \end{align*}
  thanks to the Reiteration Theorem of real interpolation theory~\cite[Theorem~3.11.15]{BerghLoefstroem1976}.
  In particular, the monotonicity of Lorentz spaces with respect to the second parameter gives  
  \begin{align}  
    (L^{\vec p_{\theta_0}}, L^{\vec p_{\theta_1}})_{\vartheta,q}
    \supset  (L^{\vec p_0,\vec r_0},L^{\vec p_1,\vec r_1})_{(1-\vartheta)\theta_0+\vartheta\theta_1,q}
    \qquad\text{for } \vec r_0,\vec r_1\in [1,\infty]^2. 
    \label{eq:Zq0infinity2} 
  \end{align}
  
  \medskip
  
  The opposite inclusion is proved with a duality argument. We get for $q\in [1,\infty),\vartheta\in (0,1)$
  \begin{align*}
    \|f\|_{(L^{\vec p_0,\vec 1},L^{\vec p_1,\vec 1})_{(1-\vartheta)\theta_0+\vartheta\theta_1,q}}
    &\simeq  \sup \left\{ T(f) :
    \|T\|_{((L^{\vec p_0,\vec 1},L^{\vec p_1,\vec 1})_{(1-\vartheta)\theta_0+\vartheta\theta_1,q})'}
    \leq  1\right\} \\
    &\simeq \sup \left\{ T(f) :
    \|T\|_{((L^{\vec p_0,\vec 1})',(L^{\vec p_1,\vec 1})')_{(1-\vartheta)\theta_0+\vartheta\theta_1,q'}}
    \leq  1\right\} \\
    &\simeq  \sup \left\{ \int_{X\times Y} fg\,d\mu :
    \|g\|_{(L^{\vecp p_0,\vec \infty},L^{\vecp p_1,\vec \infty})_{(1-\vartheta)\theta_0+\vartheta\theta_1,q'}}
    \leq 1\right\} \\
    &\stackrel{\eqref{eq:Zq0infinity2}}\les
    \sup \left\{ \int_{X\times Y} fg\,d\mu :
    \|g\|_{ (L^{\vecp p_{\theta_0}}, L^{\vecp p_{\theta_1}})_{\vartheta,q'}}
     \leq  1\right\}  \\
    &\simeq \sup \left\{ T(f) :
    \|T\|_{ ((L^{\vec p_{\theta_0}}, L^{\vec p_{\theta_1}})_{\vartheta,q})'}
     \leq  1\right\}  \\
    &\simeq \|f\|_{(L^{\vec p_{\theta_0}}, L^{\vec p_{\theta_1}})_{\vartheta,q}}.
  \end{align*}
  In the second and fifth line we used the formula for the dual of a real interpolation space  
  \cite[Theorem~3.7.1]{BerghLoefstroem1976} that is valid for $1\leq q<\infty$. 
  Proposition~\ref{prop:dual} was used in third and fifth line. The restriction $q<\infty$ can be removed by
  reiterating this family of embeddings with respect to $\vartheta$ just as above. This implies for all $\vec
  r_0,\vec r_1\in [1,\infty]^2$
    \begin{align} \label{eq:Zq01}
      (L^{\vec p_{\theta_0}}, L^{\vec p_{\theta_1}})_{\vartheta,q}
      \subset (L^{\vec p_0,\vec 1}, L^{\vec p_1,\vec 1})_{(1-\vartheta)\theta_0+\vartheta\theta_1,q}
      \subset (L^{\vec p_0,\vec r_0}, L^{\vec p_1,\vec r_1})_{(1-\vartheta)\theta_0+\vartheta\theta_1,q}. 
    \end{align}
   So \eqref{eq:Zq0infinity2},\eqref{eq:Zq01} give 
  \begin{equation} \label{eq:MixedInterpolationEquality}
     (L^{\vec p_{\theta_0}}, L^{\vec p_{\theta_1}})_{\vartheta,q}
     = (L^{\vec p_0,\vec r_0},L^{\vec p_1,\vec r_1})_{(1-\vartheta)\theta_0+\vartheta\theta_1,q}
     \quad\text{for }\vec r_0,\vec r_1\in [1,\infty]^2,
  \end{equation} 
  which finishes the proof of the theorem.
  
  \medskip
  
  To prove Corollary~\ref{cor:MixedInterpolation} we  additionally assume $\vec p_0=(p_0,p_0),\vec
  p_1=(p_1,p_1)$ where $1<p_0\neq p_1<\infty$.  Then the left hand side
  in~\eqref{eq:MixedInterpolationEquality} equals 
  $(L^{p_{\theta_0}},L^{p_{\theta_1}})_{\vartheta,q}= L^{p_{\theta},q}$ where $\theta:=
  (1-\vartheta)\theta_0+\vartheta\theta_1$.
  Since $\theta$ may take all values in $(0,1)$ for appropriate choices of $\theta_0,\theta_1,\vartheta$, we
  get 
  $$
    L^{p_\theta,q} = (L^{\vec p_0,\vec r_0},L^{\vec p_1,\vec r_1})_{\theta,q} \qquad\text{for all
    }\theta\in (0,1),\; q\in [1,\infty]
  $$
  which is all we had to show.

  \begin{rem}
     Both results require $p_0^i\neq p_1^i$ for $i=1,2$, which is needed
     for~\eqref{eq:ChenSunInequality}. In the opposite case, similar embeddings may be based on suitable
     generalizations of the interpolation inequality
     $$
       \|u\|_{p,r_\theta} \les \|u\|_{p,r_0}^{1-\theta}\|u\|_{p,r_1}^{\theta}
     $$
     for $1\leq p<\infty$, see the last line in~\cite[Theorem~5.3.1]{BerghLoefstroem1976}.    
  \end{rem}

 \section{Minkowski's inequality in integral form} \label{sec:Minkowski}

 We first state Minkowski's inequality in integral form for Lorentz spaces $L^{p,r}(Y)$. For notational
 convenience we occasionally write $L^{p,r}$ in this section. 
 This is actually a known result, see~\cite[Theorem~6.3]{BarzaKolyadaSoria2009} for the case $1<p<r\leq
 \infty$ and to \cite[p.121]{KolyadaSoria2016} for the case $1\leq r\leq p<\infty$. We present a very short self-contained
 proof that, in contrast to \cite{BarzaKolyadaSoria2009,KolyadaSoria2016}, does however not reveal precise
 information about the best constants.  

\begin{lem}\label{lem:MinkowskiI}
  If $1<p<\infty,1\leq r\leq \infty$ or $p=r\in\{1,\infty\}$ then
  $$
    \left\|\int_X f(x,\cdot)\,d\mu_X(x)\right\|_{L^{p,r}(Y)}
    \les \int_X \|f(x,\cdot)\|_{L^{p,r}(Y)} \,d\mu_X(x).
  $$
\end{lem}
\begin{proof}
  Set $Tf(y) := \int_X f(x,y)\,d\mu_X(x)$. Fubini's Theorem gives $\|Tf\|_{L^1(X)}\leq
  \|f\|_{L^1(X;L^1)}$. Moreover, we have $\|Tf\|_{L^\infty(X)}\leq \|f\|_{L^1(X;L^\infty)}$.
  Real interpolation gives for $\frac{1}{p}= \frac{1-\theta}{1}+\frac{\theta}{\infty},\theta\in (0,1)$ and
  $r\in [1,\infty]$ 
  $$
    \|Tf\|_{L^{p,r}(X)}
    \les \|f\|_{(L^1(X;L^1),L^1(X;L^\infty))_{\theta,r}}
    \les \|f\|_{L^1(X;(L^1,L^\infty)_{\theta,r})}
    = \int_X \|f(x,\cdot)\|_{p,r}\,d\mu_X(x).
  $$
  In the second estimate we used the Minkowski-type inequality for Lebesgue spaces corresponding to
  the embedding~\eqref{eq:CwikelInclusionII} for $p=1$.
\end{proof}

  We will need the following property of Peetre's $K$-functional for
  $s\geq 0$ and exponents $0<p<\infty,0<r\leq \infty$ or $p=r=\infty$:
  \begin{align} \label{eq:KfunctionalRule}
  \begin{aligned}
   &\; \|x\mapsto K(s,f(x,\cdot),A_0,A_1) \|_{L^{p,r}(X)} \\
  &=  \| x\mapsto  \inf_{g_0+g_1=f(x,\cdot)} \|g_0\|_{A_0}+s\|g_1\|_{A_1} \|_{L^{p,r}(X)}  \\
  &=   \inf_{f_0+f_1=f}  \| x\mapsto  \|f_0(x,\cdot)\|_{A_0} + s \|f_1(x,\cdot)\|_{A_1} \|_{L^{p,r}(X)}  \\
  &\simeq   \inf_{f_0+f_1=f}  \| x\mapsto  \|f_0(x,\cdot)\|_{A_0} \|_{L^{p,r}(X)} + s \|
   x\mapsto  \|f_1(x,\cdot)\|_{A_1} \|_{L^{p,r}(X)} \\
  &=   \inf_{f_0+f_1=f}  \|f_0\|_{L^{p,r}(X;A_0)}  + s \| f_1\|_{L^{p,r}(X;A_1)} \\
  &= K(s,f,L^{p,r}(X;A_0),L^{p,r}(X;A_1)).
\end{aligned}
\end{align}
In the third line the bound $\les$ only exploits  the quasi-norm property of $\|\cdot\|_{L^{p,r}(X)}$ and
$\gtrsim$ uses that $0\leq h_1\leq h_2$ on $X$ implies $\|h_1\|_{L^{p,r}(X)}\les \|h_2\|_{L^{p,r}(X)}$. The
latter fact   also explains the passage from the first to the second line. In a more general context,
\eqref{eq:KfunctionalRule} was proved in~\cite[Proposition~4.4]{ChenSun2022}.

\begin{lem}\label{lem:MinkowskiII}
  If $0<p<1,0<r\leq 1$ or $p=r=1$ then
  $$
    \int_X \|f(x,\cdot)\|_{L^{p,r}} \,d\mu_X(x)
    \les \left\|\int_X |f(x,\cdot)|\,d\mu_X(x)\right\|_{L^{p,r}}.
  $$
\end{lem}
\begin{proof}
  For $p=r\in (0,1]$ this inequality (between Lebesgue spaces) is well-known, which we will use below.
  So consider $0<p<1,0<r\leq 1$ and choose $0<q<p,\theta\in (0,1)$ with the property
  $\frac{1}{p}=\frac{1-\theta}{1}+\frac{\theta}{q}$, so
  $L^{p,r}=(L^1,L^q)_{\theta,r}$ by \cite[Theorem~5.3.1]{BerghLoefstroem1976}. We shall deduce the claimed
  estimate from the corresponding inequalities for the spaces $L^1$ and $L^q$.
  Set $d\tau(s):= s^{-\theta r-1}\,ds$. Minkowski's inequality in the Lebesgue space
  $L^{1/r}(X)$ (see Lemma~\ref{lem:MinkowskiI}) gives
  \begin{align*}
    \int_X \|f(x,\cdot)\|_{p,r} \,d\mu_X(x)
    &\les \int_X \|f(x,\cdot)\|_{(L^1,L^q)_{\theta,r}} \,d\mu_X(x) \\
    &\stackrel{\eqref{eq:DefNormInterpolationspace}}= \int_X  \left( \int_0^\infty
    K(s,|f(x,\cdot)|,L^1,L^q)^r \,d\tau(s) \right)^{\frac{1}{r}} \,d\mu_X(x)  \\
    &= \left\| x\mapsto \int_0^\infty K(s,|f(x,\cdot)|,L^1,L^q)^r \,d\tau(s)
    \right\|_{L^{\frac{1}{r}}(X)}^{\frac{1}{r}} \\
    &\les \left( \int_0^\infty \left\| x\mapsto  K(s,|f(x,\cdot)|,L^1,L^q)^r
    \right\|_{L^{\frac{1}{r}}(X)}  \,d\tau(s) \right)^{\frac{1}{r}} \\
    &= \left( \int_0^\infty \left( \int_X K(s,|f(x,\cdot)|,L^1,L^q) \,d\mu_X(x) \right)^r
     \,d\tau(s) \right)^{\frac{1}{r}} \\
   &\stackrel{\eqref{eq:KfunctionalRule}}\simeq
   \left(    \int_0^\infty K(s,|f|,L^1(X;L^1),L^1(X;L^q))^r \,d\tau(s)
    \right)^{\frac{1}{r}} \\
    &=
   \left(    \int_0^\infty \left( \inf_{f_0+f_1=|f|, f_0,f_1\geq 0} \|f_0\|_{L^1(X;L^1)} +
   s\|f_1\|_{L^1(X;L^q)}\right)^r \,d\tau(s)
    \right)^{\frac{1}{r}}.
  \intertext{
  Exploiting Minkowski's reverse inequality in the Lebesgue spaces $L^1$ and $L^q$
  we deduce}
  \int_X \|f(x,\cdot)\|_{p,r} \,d\mu_X(x)
    &\les
   \Bigg(    \int_0^\infty \Bigg( \inf_{f_0+f_1=|f|, f_0,f_1\geq 0 } \left\| \int_X f_0(x,\cdot)\,d\mu_X(x)
   \right\|_{L^1} \\
   &\qquad\qquad + s \left\| \int_X  f_1(x,\cdot)\,d\mu_X(x) \right\|_{L^q}\Bigg)^r \,d\tau(s)
    \Bigg)^{\frac{1}{r}} \\
    &= \left(    \int_0^\infty \left( \inf_{g_0+g_1=\int_X |f(x,\cdot)|\,d\mu_X(x), g_0,g_1\geq 0 }
   \|g_0\|_{L^1} + s\|g_1\|_{L^q}\right)^r \,d\tau(s) \right)^{\frac{1}{r}} \\
    &= \left(    \int_0^\infty K\left(s,\int_X |f(x,\cdot)|\,d\mu_X(x),L^1,L^q\right)^r \,d\tau(s)
    \right)^{\frac{1}{r}} \\
    &=    \left\|\int_X |f(x,\cdot)|\,d\mu_X(x)\right\|_{(L^1,L^q)_{\theta,r}}  \\
    &\les  \left\|\int_X |f(x,\cdot)|\,d\mu_X(x)\right\|_{L^{p,r}}. 
  \end{align*}
  In the fourth last line we used that for all nonnegative functions $g_0,g_1$ such that $g_0(y)+g_1(y) =
  \int_X |f(x,y)|\,d\mu_X(x)$ for all $y\in Y$ we can find nonnegative $f_0,f_1$ such that $f_0+f_1=|f|$ and
  $\int_X f_i(x,y)\,d\mu_X(x) = g_i(y)$ for $i=0,1$. For instance choose
  $f_0(x,y):= \theta(y)|f(x,y)|,f_1(x,y):=(1-\theta(y))|f(x,y)|$ where $\theta(y):= g_0(y)g(y)^{-1}
  \ind_{g(y)\neq 0}\in [0,1]$.
\end{proof}

\section{Counterexamples for Minkowski's inequality in integral form} \label{sec:MinkowskiCounter}

 We now show that the Minkowski inequalities from the previous section can only hold under the
 assumptions stated there. We always assume $0<p<\infty,0<r\leq \infty$ or $p=r=\infty$ in the following.

\subsection{\textbf{Necessity of $r\geq 1$ in Lemma~\ref{lem:MinkowskiI} and $r\leq 1$ in
Lemma~\ref{lem:MinkowskiII}}}

We only need to consider the case $p<\infty$. Take the function
$$
  f(x,y) = \ind_{x^{\alpha} y\leq 1},\qquad  x\in X:=[0,1],\; y\in Y:=\R_+,
$$
where both measure spaces are equipped with the one-dimensional Lebesgue measure. The parameter $\alpha$ will
be chosen such that $\alpha\nearrow p$. On the one hand we have
\begin{align*}
  \int_0^1  \|f(x,\cdot)\|_{p,r}\,dx
  = \int_0^1 \|\ind_{x^{\alpha}(\cdot)\leq 1}\|_{p,r}\,dx
   = \int_0^1 x^{-\frac{\alpha}{p}} \,dx \,\|\ind_{[0,1]}\|_{p,r}
   \simeq (p-\alpha)^{-1}.
\end{align*}
 On the other hand,
 \begin{align*}
  \left\|\int_0^1 f(x,\cdot)\,dx \right\|_{p,r}
  &= \left\|\int_0^1 \ind_{x^\alpha (\cdot)\leq 1} \,dx \right\|_{p,r}
  \;=\; \left\| \min\{1,(\cdot)^{-\frac{1}{\alpha}}\}  \right\|_{p,r} \\
  &= \left\| \rho |\{\min\{1,(\cdot)^{-\frac{1}{\alpha}}\}\geq \rho\}|^{\frac{1}{p}}
  \right\|_{L^r(\R_+,\rho^{-1}\,d\rho)} \\
  &= \left\| \rho^{1-\frac{\alpha}{p}} \right\|_{L^r([0,1],\rho^{-1}\,d\rho)} \\
  &\simeq  (p-\alpha)^{-\frac{1}{r}}.
\end{align*}
As a consequence,
$$
  \frac{\left\|\int_0^1 f(x,\cdot)\,dx \right\|_{p,r}}{\int_0^1  \|f(x,\cdot)\|_{p,r}\,dx}
  \simeq (p-\alpha)^{1-\frac{1}{r}}.
$$
This proves the necessity of $r\geq 1$ (Lemma~\ref{lem:MinkowskiI}) respectively $r\leq 1$
(Lemma~\ref{lem:MinkowskiII}).

\medskip

\subsection{\textbf{Necessity of $p\geq 1$ in Lemma~\ref{lem:MinkowskiI}  and $p\leq 1$
in Lemma~\ref{lem:MinkowskiII}}}
Consider
$$
  f(x,y) = \ind_{x\in\{1,\ldots,N\}}\, \phi(y-x),\qquad x\in X:=\N,\; y\in Y:=\R 
$$
equipped with the Lebesgue measure respectively counting measure. We assume $\phi\in L^{p,r}(\R),\phi\geq 0$
and $\supp(\phi)\subset [-\frac{1}{2},\frac{1}{2}]$. Then
$$
  \sum_{x=1}^N \|f(x,\cdot)\|_{p,r}
 = N \|\phi\|_{p,r}.
$$
On the other hand, exploiting $\supp(\phi)\subset [-\frac{1}{2},\frac{1}{2}]$ we find
\begin{align*}
   \left\| \sum_{x=1}^N f(x,\cdot) \right\|_{p,r}
   &= \left\| \rho \mu\left(\left\{\sum_{x=1}^N \phi(\cdot-x)\geq
   \rho\right\}\right)^{\frac{1}{p}}\right\|_{L^r(\R_+,\rho^{-1}\,d\rho)} \\
   &= \|\rho (N\mu(\{\phi\geq  \rho\})^{\frac{1}{p}}\|_{L^r(\R_+,\rho^{-1}\,d\rho)}  \\
   &=  N^{\frac{1}{p}} \|\phi\|_{p,r}
\end{align*}
and the claim follows from $N\to\infty$.

\subsection{\textbf{Minkowski fails in $L^{1,r},1<r\leq \infty$}}

We fix $\beta\in (\frac{1}{r},1)$ and define for $N\in\N$ (that we shall send to infinity) the function 
$$
    f(x,y) = \ind_{|x|\in\{0,\ldots,N\}} \phi(y-x),\qquad
    \phi(z) = (2+|z|)^{-1} \log(2+|z|)^{-\beta} \quad (z\in\R)
$$
where  $x\in X:=\N,y\in Y:=\R$. The spaces $X,Y$ are equipped with the counting measure and the Lebesgue
measure, respectively. A calculation gives
$$
  \sum_{|x|\leq N}  \| f(x,\cdot)\|_{1,r}
   = (2N+1)   \| \phi\|_{1,r}
   \simeq N.
$$
Note that $\beta>\frac{1}{r}$ implies that $\| \phi\|_{1,r}$ is finite.
On the other hand, we obtain  for $I_l=[l-\frac{1}{2},l+\frac{1}{2}]$
\begin{align*}
    \sum_{|x|\leq N}  f(x,y)
    &= \sum_{|x|\leq N}  (2+|y-x|)^{-1} \log(2+|y-x|)^{-\beta} \\
    &\geq \sum_{|l|\leq N}  \ind_{I_l}(y)  \sum_{|x|\leq N}  (2+|y-x|)^{-1} \log(2+|y-x|)^{-\beta} \\
    &\simeq \sum_{|l|\leq N}  \ind_{I_l}(y) \sum_{|x|\leq N}   (2+|l-x|)^{-1} \log(2+|l-x|)^{-\beta} \\
    &\gtrsim \sum_{|l|\leq N}  \ind_{I_l}(y) \sum_{m=2}^N   m^{-1} \log(m)^{-\beta} \\
    &\simeq   \ind_{[-N,N]}(y) \int_2^N t^{-1} \log(t)^{-\beta}\,dt \\
    &\simeq \ind_{[-N,N]}(y) \log(N)^{1-\beta}.
  \end{align*}
This implies
$$
  \left\|\sum_{|x|\leq N}  f(x,\cdot)\right\|_{1,r}
  \gtrsim   \log(N)^{1-\beta}  \| \ind_{[-N,N]} \|_{1,r}
  \gtrsim  N \log(N)^{1-\beta},
$$
which grows faster than $N$ as $N\to\infty$. So Minkowski's inequality from
Lemma~\ref{lem:MinkowskiI} does not hold in $L^{1,r}$ when $r>1$.

\subsection{\textbf{Reverse Minkowski fails in $L^{1,r},0<r<1$}}

Now take $N\in\N$ and $\beta>\frac{1}{r}>1$. We will consider $\beta$ close to
$\frac{1}{r}$ and $N\to\infty$. Then the same function $f$ as above satisfies
$$
  \sum_{|x|\leq N}  \| f(x,\cdot)\|_{1,r}
   = (2N+1)   \| \phi\|_{1,r}
   \simeq N (r\beta -1)^{-\frac{1}{r}}
$$
as well as
  \begin{align*}
    \sum_{|x|\leq N}  f(x,y)
    &\simeq \sum_{l\in\Z}  \ind_{I_l}(y) \sum_{|x|\leq N}   (2+|l-x|)^{-1} \log(2+|l-x|)^{-\beta} \\
    &\les \ind_{[-2N,2N]}(y) \sum_{|x|\leq 3N}   (2+|x|)^{-1} \log(2+|x|)^{-\beta}  \\
    &+ \sum_{|l|\geq 2N}  \ind_{I_l}(y) \sum_{|x|\leq N}   (2+|l|)^{-1} \log(2+|l|)^{-\beta}  \\
    &\simeq
    \ind_{[-2N,2N]}(y) \int_2^{3N+2} t^{-1} \log(t)^{-\beta}\,dt
    + N  \sum_{|l|\geq 2N+2}  \ind_{I_l}(y)  |l|^{-1} \log(|l|)^{-\beta}  \\
    &\simeq \ind_{[-2N,2N]}(y)
    + N  \sum_{|l|\geq 2N+2}  \ind_{I_l}(y) |l|^{-1} \log(|l|)^{-\beta}.
  \end{align*}
  Next we exploit that the inverse of $t\mapsto t^{-1}\log(t)^{-\beta}$ behaves like $\rho\mapsto
  \rho^{-1}|\log(\rho)|^{-\beta}$ as $t\to\infty$ or $r\to 0^+$, respectively. 
  Hence, we get a positive constant $C$ such that we have for large $N\in\N$
  \begin{align*}
    \left\| \sum_{|x|\leq N}  f(x,\cdot) \right\|_{1,r}
    &\les N + N
    \left\|\sum_{|l|\geq 2N+2}  \ind_{I_l}(y) |l|^{-1} \log(|l|)^{-\beta}\right\|_{1,r} \\
    &= N +    N
     \left(\int_0^\infty \rho^{r-1} \left( \sum_{|l|\geq 2N+2}  \ind_{|l|^{-1}
    \log(|l|)^{-\beta} \geq \rho} \right)^r\,d\rho\right)^{\frac{1}{r}} \\
    &\les N +    N
     \left(\int_0^{CN^{-1}\log(N)^{-\beta}} \rho^{r-1} \left( \int_0^{C \rho^{-1}
    \log(\rho)^{-\beta}} 1\,dl \right)^r\,d\rho\right)^{\frac{1}{r}} \\
    &\les
    N +    N
     \left(\int_0^{CN^{-1}} \rho^{r-1} \left(\rho^{-1}|\log(\rho)|^{-\beta}
     \right)^{r}\,d\rho\right)^{\frac{1}{r}}    \\
    &\simeq N + N \left(\int_{\log(NC^{-1})}^\infty t^{-r\beta}  \,dt\right)^{\frac{1}{r}} \\
    &\simeq N +   N\log(N)^{\frac{1}{r}-\beta} (r\beta-1)^{-\frac{1}{r}}.
  \end{align*}
 This implies
  \begin{align*}
    \frac{ \left\| \sum_{|x|\leq N}  f(x,\cdot) \right\|_{1,r}}{
    \sum_{|x|\leq N}  \|   f(x,\cdot)\|_{1,r} }
    \simeq \frac{N +   N\log(N)^{\frac{1}{r}-\beta} (r\beta-1)^{-\frac{1}{r}}}{N (r\beta -1)^{-\frac{1}{r}}}
    \simeq  (r\beta -1)^{\frac{1}{r}} +   \log(N)^{\frac{1}{r}-\beta}.
  \end{align*}
   We first choose $\beta$ sufficiently close to $\frac{1}{r}$ and then send $N$ to infinity. In this
   way, the quotient can be made arbitrarily small.  So the reverse Minkowski inequality from Lemma~\ref{lem:MinkowskiII} does not hold
   in $L^{1,r}$ with $r\in (0,1)$.

\section{Proof of Theorem~\ref{thm:MixedInterpolationII}}

We first prove (i). So let us assume $0<q<p<\infty, q\leq r\leq \infty$ or $0<p=q=r\leq \infty$
and we want have to show
  $$
    (L^{p,r}(X;A_0),L^{p,r}(X;A_1))_{\theta,q}
    \subset L^{p,r}(X;(A_0,A_1)_{\theta,q}).
  $$
We first assume $q<\infty$, set $d\tau(s)=s^{-\theta q-1}\,ds$. Minkowski's inequality
(Lemma~\ref{lem:MinkowskiI}) yields
\begin{align*}
  \|f\|_{L^{p,r}(X;(A_0,A_1)_{\theta,q})}^q
  &= \| \|f(x,\cdot)\|_{(A_0,A_1)_{\theta,q}}  \|_{L^{p,r}(X)}^q \\
  &= \left\| \left( \int_0^\infty K(s,f(x,\cdot),A_0,A_1)^q \,d\tau(s) \right)^{\frac{1}{q}}
  \right\|_{L^{p,r}(X)}^q  \\
  &= \left\|  \int_0^\infty K(s,f(x,\cdot),A_0,A_1)^q \,d\tau(s)  \right\|_{L^{\frac{p}{q},\frac{r}{q}}(X)}
  \\
  &\les  \int_0^\infty\|  K(s,f(x,\cdot),A_0,A_1)^q \|_{L^{\frac{p}{q},\frac{r}{q}}(X)} \,d\tau(s)  \\
  &=  \int_0^\infty \|K(s,f(x,\cdot),A_0,A_1)  \|_{L^{p,r}(X)}^q  \,d\tau(s) \\
  &\stackrel{\eqref{eq:KfunctionalRule}}= \int_0^\infty
  K(s,f,L^{p,r}(X;A_0),L^{p,r}(X;A_1))^q \,d\tau(s) \\
  &= \|f\|_{(L^{p,r}(X;A_0),L^{p,r}(X;A_1))_{\theta,q}}^q.
\end{align*}
In the case $p=q=r=\infty$ the claim holds due to  
\begin{align*}
  \left\| \sup_{s>0} s^{-1} K(s,f(x,\cdot),A_0,A_1)\right\|_{L^{\infty}(X)}
  &= \sup_{s>0,x\in X} s^{-1} K(s,f(x,\cdot),A_0,A_1) \\
  &= \sup_{s>0} s^{-1}  \left\| K(s,f(x,\cdot),A_0,A_1)\right\|_{L^{\infty}(X)}.
\end{align*}
 
Now assume $0<p< q\leq \infty,0<r\leq q$ or $0<p=q=r\leq \infty$. Claim (ii) amounts to proving the
opposite inclusion
$$
    (L^{p,r}(X;A_0),L^{p,r}(X;A_1))_{\theta,q}
    \supset L^{p,r}(X;(A_0,A_1)_{\theta,q}).
  $$
In the case $q<\infty$ this is achieved just as above by replacing the classical Minkowski inequality
by the reverse one from Lemma~\ref{lem:MinkowskiII}, which requires $0<\frac{p}{q}<1,0<\frac{r}{q}\leq 1$ or
$0<r=p=q<\infty$. In the case $q=\infty$ one argues similarly by exploiting
$$
  \left\| \sup_{s>0} s^{-1} K(s,f(x,\cdot),A_0,A_1)\right\|_{L^{p,r}(X)}
  \geq \sup_{s>0} s^{-1}  \left\| K(s,f(x,\cdot),A_0,A_1)\right\|_{L^{p,r}(X)}.
$$

\section{An embedding}
    
  The computation of interpolation spaces for mixed Lorentz spaces simplifies a lot once embeddings of the
  latter into ordinary Lorentz spaces are known. The following Lemma, which 
  was proved in~\cite[Theorem~2.4]{ChenSun2020} in the special case
  $r=\infty$, provides one such result.  

   \begin{lem} \label{lem:inclusions}
     Let $(X,\mu_X),(Y,\mu_Y)$ be $\sigma$-finite measure spaces. Then:
     \begin{itemize}
       \item[(i)] $L^p(X;L^{p,r}(Y)) \subset L^{p,r}(X\times Y)$ if $0<p\leq r\leq \infty$,
       \item[(ii)] $L^p(X;L^{p,r}(Y)) \supset L^{p,r}(X\times Y)$ if $0<r\leq p\leq \infty$.
     \end{itemize}
   \end{lem}
   \begin{proof}
     Let $f:X\times Y\to\R$ be measurable. Then $\mu=\mu_X\otimes \mu_Y$, Cavalieri's principle and
     Minkowski's inequality (Lemma~\ref{lem:MinkowskiI}) imply for $0<p\leq r< \infty$
     \begin{align*}
       \|f\|_{L^{p,r}(X\times Y)}^r
       &= \|\rho\, \mu(\{|f|\geq \rho\})^{\frac{1}{p}} \|_{L^r(\R_+,\rho^{-1}\,d\rho)}^r \\
       &= \left\|\rho \left( \int_X \mu_Y(\{|f(x,\cdot)|\geq \rho\})\,d\mu_X(x)\right)^{\frac{1}{p}}
       \right\|_{L^r(\R_+,\rho^{-1}\,d\rho)}^r    \\
       &= \left\|  \int_X \mu_Y(\{|f(x,\cdot)|\geq \rho\})\,d\mu_X(x)
       \right\|_{L^{\frac{r}{p}}(\R_+,\rho^{r-1}\,d\rho)}^{\frac{r}{p}}    \\
       &\les  \left( \int_X \left\| \mu_Y(\{|f(x,\cdot)|\geq \rho\})
       \right\|_{L^{\frac{r}{p}}(\R_+,\rho^{r-1}\,d\rho)} \,d\mu_X(x)  \right)^{\frac{r}{p}}   \\
       &=  \left( \int_X \left\| \rho\, \mu_Y(\{|f(x,\cdot)|\geq \rho\})^{\frac{1}{p}}
       \right\|_{L^r(\R_+,\rho^{-1}\,d\rho)} \,d\mu_X(x)  \right)^{\frac{r}{p}}   \\
       &=  \left( \int_X \|f(x,\cdot)\|_{L^{p,r}(Y)}^p  \,d\mu_X(x)  \right)^{\frac{r}{p}}   \\
       &=  \|f\|_{L^p(X;L^{p,r}(Y))}^r.
     \end{align*}
     This proves (i) for $r<\infty$ and the proof for $r=\infty$ is similar. Using the reverse inequality from
     Lemma~\ref{lem:MinkowskiII} instead of Lemma~\ref{lem:MinkowskiI} we get (ii).
   \end{proof}

\section*{Acknowledgments}

Funded by the Deutsche Forschungsgemeinschaft (DFG, German Research Foundation) -- Project-ID
258734477 -- SFB~1173.
    
\bibliographystyle{abbrv}
\bibliography{biblio}

\end{document}